\documentclass[reqno,oneside]{amsart}         
\usepackage{mathrsfs}    
\usepackage{verbatim}
\usepackage[backref,colorlinks]{hyperref}  

\usepackage{subcaption}
\usepackage[english]{babel}
\usepackage[utf8]{inputenc}
\usepackage{amsmath,amssymb}
\usepackage{algorithm}
\usepackage{algpseudocode}
\usepackage{geometry}
\usepackage{xspace}

\usepackage{amsmath,amsfonts,amssymb,amsopn,amsthm}
\usepackage{enumerate}
\theoremstyle{plain}
\usepackage{array}

\newtheorem{theorem}{Theorem}%[theorem]
\newtheorem{lemma}[theorem]{Lemma}

\newtheorem{cor}[theorem]{Corollary}

\newcommand{\FF}{\mathcal F}

\title{Star Coloring of Hypergraphs}


\author[L. \"Ozkahya]{Lale \"Ozkahya}
\address{Hacettepe University, Department of Computer Engineering, Beytepe 06810 Ankara, Turkey.}
\email{lale.ozkahya@gmail.com}

\author[Polat Sar\i yerl\.i]{Polat Sar\i yerl\.i}
\address{Hacettepe University, Department of Computer Engineering, Beytepe 06810 Ankara, Turkey.}
\email{polatsariyerli@gmail.com}

\begin{document}

\begin{abstract}
We study a generalization of the star coloring problem on hypergraphs. For a family of connected subhypergraphs $\FF$, we define an $\FF$-coloring of a hypergraph as a coloring  avoiding monochromatic hyperedges and any 2-colored member of $\FF$. 
We let $\chi^r_{\FF}(d)$ be the maximum of the minimum number of colors needed for an $\FF$-coloring of an $r$-uniform  hypergraph with maximum degree $d.$ We show bounds for $\chi^r_{\FF}(d)$, that also yield results on star and acyclic coloring problem on hypergraphs.  
\end{abstract}


%\item 
%\item 
%\item 

\maketitle

\section{Introduction} \label{sec:introduction}

Hypergraph coloring is widely studied, such as in~\cite{bohman2010coloring,cooper2015list,cooper2016coloring,cooper2017sparse,frieze2008chromatic,frieze2011randomly,frieze2013coloring,kostochka2009coloring,kostochka2012conflict,li2022chromatic}. A hypergraph is called \emph{$r$-uniform} if all its hyperedges have exactly $r$ vertices. 
Erd{\H o}s and Lov\'asz \cite{erd1975problems} proved that the vertex set of any $(r+1)$-uniform hypergraph with maximum degree $\Delta$ can be colored with $c\Delta^{1/r}$, for some constant $c,$ such that there is no monochromatic hyperedge. This result is proved by an application of Lov\'asz Local Lemma, which is also introduced in the same paper. 
In this paper, we aim to study a generalization of star coloring problem 
on hypergraphs. 
%two graph coloring problems, namely star coloring and acyclic coloring, to hypergraphs. 
The {\it star-coloring} of a graph, introduced by Gr\"unbaum \cite{grunbaum1973acyclic}, is defined as the proper vertex coloring, where only stars are allowed to be bicolored (2-colored). 
In the same paper, acyclic coloring of graphs is introduced, defined as a proper vertex coloring without bicolored cycles. 
%The star coloring of graphs is introduced by Gr\"unbaum \cite{grunbaum1973acyclic}, who proved that a graph with maximum degree 3 has an acyclic coloring with 4 colors. 
The star and acyclic coloring problems are shown to be NP-complete by Albertson et al.~\cite{albertson2004coloring} and Kostochka~\cite{kostochka1978upper}, respectively.

We consider a generalization of these colorings to hypergraphs, by letting $\FF$ be a fixed family of connected subhypergraphs and 
{\it $\FF$-coloring of a hypergraph} be a vertex coloring with no monochromatic hyperedge and no bicolored member of $\FF$. Let $K_n^r$ denote the complete $r$-uniform hypergraphs on $n$ vertices and $\Delta(H)$ denote the maximum vertex degree in a hypergraph $H.$ 
For a hypergraph $H$, we let $\chi_{\FF}(H)$ be the minimum number of colors needed in an $\FF$-coloring of $H$ and $\chi_{\FF}^r(d):=\max\{\chi_{\FF}(H): H\subset K_n^r\text{ and } \Delta(H)=d\}.$  
We prove the following general upper bound on $\FF$-coloring of hypergraphs.
\begin{theorem}\label{thm:general_uppbnd}
  %Let $H$ be a r-uniform hypergraph of maximum degree $\Delta$, with vertex set $V=V(H)$ and edge set $E = E(H)$. If a coloring of $H$ avoids bichromatic copies of connected subgraphs in a family $\FF$, then 
  For $r\geq 2$, let $\FF$ be a family of connnected hypergraphs, 
$2\leq m=\max_{F\in \FF}|F|$, $r+1\leq a=\min_{F\in \FF}|V(F)|$ and 
  $t=\min\{m, a-1\}.$ Then $\chi_\FF^r(d)\leq \lceil {( 2^{3a-2}m^t (rd)^m)}^\frac{1}{a-2} \rceil.$  
\end{theorem} 
Let us call 
a triangle with edges $E_1, E_2, E_3$ such that $\cap_{i=1}^3 E_i=\emptyset$ and 
$|E_1\cap E_j|=\lfloor r/2\rfloor$ for $j=2,3,$  and $|E_2\cap E_3|=\lceil r/2\rceil$, a {\it balanced triangle.} For $\FF$-coloring, we prove the following result. This result provides a lower bound for star and acyclic colorings of hypergraphs as discussed in Section~\ref{sec:appstar}. 
\begin{theorem}\label{thm:lwrbnd}
    Let $\FF$ be a family of connected $r$-uniform hypergraphs, also having the balanced triangle as a member.  
    Then, for $r\geq 2$, $\chi_{\FF}^r(d)=\Omega ((d^3/\ln d)^\frac{1}{3r-4})$. 
\end{theorem} 
In Section~\ref{sec:appstar}, we discuss the implications of Theorems~\ref{thm:general_uppbnd} and~\ref{thm:lwrbnd} on star and acyclic coloring of hypergraphs. In Section~\ref{sec:proofs}, we provide proofs of these theorems. 

\section{Star and Acyclic Coloring of Hypergraphs}\label{sec:appstar}
%We generalize the star and acyclic coloring to hypergraphs in the following sense. 

%{\it Star coloring on hypergraphs:} 
A {\it star} hypergraph is a set of edges $E_1,\dots,E_m$, $m\geq 2$, such that $\cap_i E_i\neq \emptyset.$  We call a subhypergraph with hyperedges $E_0,\dots, E_{m-1}$ a cycle of length $m$ if $E_i\cap E_{i+1}\neq \emptyset$ and $E_{i-1}\cap E_i \cap E_{i+1} = \emptyset$ for each $i$ modulo $m.$  We let a {\it star (acyclic, resp.) coloring of a hypergraph} $H$ be defined as a coloring with no monochromatic hyperedge, where only bicolored subhypergraphs are allowed to be stars (cycles of length at least three, resp.). We let the {\it star (acyclic, resp.) chromatic number of H}, $\chi_s(H)$ ($\chi_a(H)$, resp.) be the minimum number of colors needed in a star (acyclic, resp.) coloring of $H$. We let $\chi_s^r(d)=\max\{\chi_s(H): H\subset K_n^r\text{ and }\Delta(H)=d\}$ and $\chi_a^r(d)=\max\{\chi_a(H): H\subset K_n^r\text{ and }\Delta(H)=d\}$ 
By our definition, $\chi_a^r(d)\leq \chi_s^r(d).$ 

For graphs, the star coloring problem corresponds to an $\FF$-coloring with $\FF=\{P_4\}$. Thus, using Theorem~\ref{thm:general_uppbnd} with $m=3$ and $a=4$ implies the same asymptotic upper bound $\chi_s^2(d)=O(d^{\frac{3}{2}})$ shown by Fertin et al.~\cite{fertin2004star}. 
A hypergraph is {\it linear} if every pair of edges intersect in at most one vertex. 
When $H$ is a linear $r$-uniform hypergraph with maximum degree $d$, then Theorem~\ref{thm:general_uppbnd} implies that $\chi_s(H)\leq \lceil {( 2^{3r-8}3^3 (rd)^3)}^\frac{1}{3r-5} \rceil,$ where Theorem~\ref{thm:general_uppbnd} is applied by letting $\FF$ consist of a triangle and a path on three edges, as these 
are the only minimal nonstar linear subhypergraphs. 
%thus $m=3$ and $a=3r-3.$  

For nonlinear hypergraphs, we consider a stronger coloring condition for finding an upper bound. A star with hyperedges $E_1,\dots, E_m$ is called a {\it sunflower} if 
for all $i\neq j$, $1\leq i,j\leq m,$ $E_i\cap E_j=C$ for some set $C\neq \emptyset.$  
We consider a vertex coloring of hypergraphs without monochromatic hyperedges 
allowing only sunflowers to be bicolored, as this yields also a star coloring. 
%An $\FF$-coloring of hypergraphs, where $\FF$ consists of all minimal subhypergraphs that are not sunflowers, is also a star coloring. 
Thus applying Theorem~\ref{thm:general_uppbnd}, where $\FF$ consists of subhypergraphs with three hyperedges that are not sunflowers, 
we obtain the following upper bound.  

\begin{cor}\label{thm:upperbnd_star}
For $r\geq 3$, $\chi_s^r(d)\leq \lceil {( 2^{3r+1}3^3 (rd)^3)}^\frac{1}{r-1} \rceil.$
\end{cor}

Alon, McDiarmid, and Reed~\cite{alon1991acyclic} showed that every graph has an acyclic coloring with $O(\Delta^\frac{4}{3})$ colors.   Recently, there have been some improvements in the constant factor of the upper bound in~\cite{esperet2013acyclic,goncalves:lirmm-01233456,ndreca2012improved}  by using the entropy compression method. We obtain the result below for $r\geq 3$, using Theorem~\ref{thm:general_uppbnd} with
$\FF$ consisting of paths with three edges and all triangles, hence letting $m=3$ and $a=\lceil 3r/2\rceil$ (given by the smallest order triangle). This bound slightly improves the upper bound for $\chi_a^r(d)$ in Corollary~\ref{thm:upperbnd_star}.
\begin{cor}\label{thm:upperbnd_acyclic}
For $r\geq 3$, $\chi_a^r(d)\leq 2^{9r/2-2}3^3 (rd)^{\frac{6}{3r-4}}.$
\end{cor}

%\subsection{Lower Bound}
The following observation on the star coloring of the complete $r$-uniform hypergraph $K_n^r$ provides a lower bound for $\chi_s^r(d).$
In a star coloring of $K_n^r$, no pair of color classes contains $K_{r+1}^r$ as a subhypergraph as this is not a star. Consider a star coloring of $K_n^r$ 
with $x$ colors and let $V_1$ and $V_2$ be the two largest color classes. 
 By an averaging argument, $|V_1\cup V_2|\geq 2n/x.$ As $V_1\cup V_2$ is a bichromatic set, $2n/x< r+1$, as $r+1$ is the smallest order of a nonstar subgraph. Also, $\Delta = \binom{n-1}{r-1}\leq \left(\frac{en}{r-1}\right)^{r-1}$. These together yield $\frac{2(r-1)\Delta^{\frac{1}{r-1}}}{e(r+1)}<x.$ Thus, 
$\chi_s^r(d)\geq \frac{2(r-1)}{e(r+1)}\Delta^{\frac{1}{r-1}}.$   
Theorem~\ref{thm:lwrbnd} implies the following lower bound for $\chi_s^r(d)$ that 
improves the lower bound above and generalizes   
the result by Fertin et al.~\cite{fertin2004star} for $\chi^{2}_s(d)$.    
\begin{cor}\label{cor:lowerbnd}
For $r\geq 3$, $\chi^{r}_s(d) = \Omega ((d^3/\ln d)^\frac{1}{3r-4})$ and 
$\chi^{r}_a(d) = 
\Omega ((d^3/\ln d)^\frac{1}{3r-4})$
%\left(\frac{3^{r-5}(r-1)^{3r-3}}{r^22^{r+7}e^{3r-3}}\right)}
\end{cor}

%Alon, McDiarmid, and Reed~\cite{alon1991acyclic} proved that there exist graphs with maximum degree $\Delta$ that have acyclic coloring using $\Omega((\Delta^{\frac{4}{3}})/(log \Delta)^\frac{1}{3})$ colors. 

\section{Main Results}\label{sec:proofs}

%\subsection{Preliminaries} 

In our proofs, we use the following version of the Lov\'asz Local Lemma~\cite{erdHos1973problems}. 
\begin{theorem}[General Lov\'asz Local Lemma] \label{GeneralLovaszLocal} \cite{erdHos1973problems}
 Suppose that $G=(V,E)$ is a dependency graph for the events $A_1,A_2,...,A_n$ and suppose there are real numbers $y_1,y_2,...,y_n$ such that $0 \leq y_i < 1$ and 
\begin{equation}\label{eq:LLL}
 Pr[A_i] \leq y_i \prod_{(i,j) \in E} (1-y_j)
 \qquad \text{for all $1 \leq i \leq n$.}
 \end{equation}
 Then $Pr[\bigwedge_{i=1}^n \Bar{A_i}] \geq \prod_{i=1}^n (1-y_i)$. In particular, with positive probability no event $A_i$ holds.
\end{theorem}

We also use the following tail inequality for the binary distribution 
$\mathrm{Bin}(n,p).$

\begin{lemma}[Chernoff bound, Eq.2.5 in \cite{janson2011random}]
\label{lem:chernoff}
Let $X\sim\mathrm{Bin}(n,p)$. For every $t\ge0$,
\begin{equation}\label{eq:chernoff}
  \Pr\bigl[X \ge np + t\bigr]
  \;\le\;
  \exp \Bigl(-\frac{t^2}{2\,(np + t/3)}\Bigr).
\end{equation}
\end{lemma}

\subsection{Proof of Theorem \ref{thm:general_uppbnd}}

%\begin{proof}[Proof of Theorem~\ref{thm:general_uppbnd}]
Consider an $r$-uniform hypergraph $H=(V,E)$ with maximum degree $\Delta.$ 
Let $x= \lceil {( 2^{3a-2}m^t (r\Delta)^m)}^\frac{1}{a-2} \rceil$ with $t=\min\{m, a-1\}$, and let us color $V$ by assigning each vertex a color chosen uniformly at random from $C = \{1, 2, 3, \ldots, x\}$. We show that with nonzero probability, a coloring of $H$ exists that has none of the bad events described below:    
\begin{itemize}
    \item Type 1: For a given edge $e\in E$, bad event $A_{(M,e)}$ is the event that $e$ is monochromatic. 
    \item Type 2: For any $P\in \FF$, $A_{(P)}$ is the event that $P$ is bicolored or monochromatic.
\end{itemize}
Let $D$ be the dependency graph, where the bad events above comprise the vertex set and the edge set consists of the intersecting pairs of bad events. 

%\begin{enumerate}[i)]    
%    \item $A_{(M,e_1)}$ and $A_{(M,e_2)}$, if $e_1\cap e_2 \neq \emptyset$;
%    \item $A_{(M,e)}$ and $A_{(P)}$, if $e \cap P \neq \emptyset$;
%    \item $A_{(P_1)}$ and $A_{(P_2)}$, if $P_1 \cap P_2 \neq \emptyset$.
%\end{enumerate}
 
For any $P\in \FF$ with edges $e_1,\dots,e_i$, 
a vertex $v$ appears in at most $i\Delta(r\Delta)^{i-1}=\frac{i}{r}(r\Delta)^i$ copies of $P.$ As $m$ is the maximum number of edges 
observed on a member in $\FF$, the number of members of $\FF$ that contains $v$ is at most 
\[
\sum_{i=2}^m \frac{i}{r} (r\Delta)^i \leq \frac{r\Delta}{r}\frac{(1-(m+1)(r\Delta)^m+m(r\Delta)^{m+1})}{(1-r\Delta)^2} \leq 
\frac{4m(r\Delta)^{m+2}}{r (r\Delta)^2} = \frac{4m}{r} (r\Delta)^m
\]
Thus, for a vertex of type \( i \), an upper bound for the number of its neighbors of type \( j \), denoted by \(D_{i,j}\),  are listed  in row $i$ and column $j$ of the following table:
    \begin{table}[h!]
\centering
\begin{tabular}{|c|c|c|}
\hline
 & 1 & 2 \\
\hline
1 & $r\Delta$ & $4m(r\Delta)^m$\\
\hline
2 & $mr\Delta$ & $4m^2(r\Delta)^m$ \\
\hline
\end{tabular}
\end{table}

We have the following upper bounds on the probabilities of bad events: 
\begin{itemize}
    \item For each event of Type 1, $A_{(M, e)}$, 
    $Pr[A_{(M,e)}] = \frac{1}{x^{r-1}},$
    \item For each event of Type 2, $A_{(P)}$, 
    $Pr[A_{(P)}] \leq \frac{{x\choose 2}2^{|V(P)|}}{x^{|V(P)|}}
    \leq 2 {\left( \frac{2}{x} \right)}^{a-2},$
\end{itemize}    
where $a=\min_{F\in \FF}|V(F)|$. 
In order to apply the Lov\'asz Local Lemma, we let $y_1=2Pr[A_{(M,e)}]$ and $y_2=2Pr[A_{(P)}]$ for all events $A_{(M,e)}$ and $A_{(P)}$. %We observe that these positive values for $y_1$, $y_2$ are not greater than 1, satisfying the assumption in Lovasz Local Lemma. 
Lovasz Local Lemma requires that: 
\begin{equation}\label{eq:y_m_requirement}
    \mathbb{P}[A_{(M,e)}] \leq y_1 (1 - y_1)^{r\Delta} (1 - y_2)^{4m (r\Delta)^m}
\end{equation}
\begin{equation}\label{eq:y_i_requirement}
\mathbb{P}[A_{(P)}] \leq y_2 (1 - y_1)^{mr\Delta} (1 - y_2)^{4m^2 (r\Delta)^m}
\end{equation}
where \eqref{eq:y_i_requirement} implies \eqref{eq:y_m_requirement}. 
Thus, we need to show that 
\begin{equation}\label{eq:subs_yi}
\frac{1}{2} \leq \left( 1 - \frac{2}{x^{r-1}} \right)^{m r \Delta} 
{\left(1 - 4{\left( \frac{2}{x} \right)}^{a-2}\right)}^{4m^2 (r\Delta)^m}
\end{equation}

We observe that the parantheses on the right in \eqref{eq:subs_yi} are nonnegative. So, we use Bernoulli's inequality:
$(1 + t)^n \geq 1 + n t,$ $\forall t \geq -1, \, \forall n > 1.$ Thus, showing the following yields~\eqref{eq:subs_yi}.
\begin{equation}\label{final-result}
\frac{1}{2} \leq \left(1 - \frac{2m r \Delta }{x^{r-1}}\right) \cdot  \left(1 - 16 m^2 (r\Delta)^m  {\left( \frac{2}{x} \right)}^{a-2}   \right) 
\end{equation}

Below, we substitute  $x=\lceil {( 2^{3a-2}m^t (r\Delta)^m)}^\frac{1}{a-2}\rceil$ with $t=\min\{m, a-1\}$.
Since each $F\in \FF$ is a connected subhypergraph 
$a\leq r+(m-1)(r-1)= m(r-1)+1$ and $m^{t\frac{r-1}{a-2}}>m.$ 
This, together with $a\geq r+1\geq 3$ and $m\geq 2$ yield 
\begin{equation}\label{eq:y_M-upper-bound-2}
  \frac{2m r \Delta }{x^{r-1}} \leq  
  \frac{2m r \Delta}{(2^{3a-2}m^t (r\Delta)^m)^\frac{r-1}{a-2}} \leq 
  \frac{2}{2^{3r-3}}\frac{m}{m^{t\frac{r-1}{a-2}}} 
  \frac{r \Delta}{(r\Delta)^{m\frac{r-1}{a-2}}}\leq 
  \frac{1}{4}  
\end{equation}
and
\begin{equation}\label{eq:y_P-upper-bound}
16 m^2 (r\Delta)^m  {\left( \frac{2}{x} \right)}^{a-2} \leq 
\frac{2^{a+2} m^2 (r\Delta)^m}{2^{3a-2}m^t (r\Delta)^m} 
    \leq \frac{1}{4}
\end{equation}
These show that \eqref{final-result} holds. The constants in $x$ can be further optimized. 
%\end{proof}

%\section{Lower Bound}\label{sec:lowerbound}

\subsection{Proof of Theorem \ref{thm:lwrbnd}}\label{sec:lowerbound}

%We show a general bound on star chromatic number for all hypergraphs. 
With Lemma \ref{lem_linearpath}, we aim to show that with high probability, there is a bichromatic balanced triangle in Erd{\H o}s-Renyi hypergraph $G^r(n,p),$ where each $r$-subset of $n$ vertices is a hyperedge  with probability $p.$

\begin{lemma}\label{lem_linearpath}
Let $p$ be a real number satisfying 
${\left(\frac{r^2\ln(n)}{n} \cdot \frac{25\cdot 2^{r+6}}{3^{r-1}}\right)}^\frac{1}{3} \leq p \leq 1.$ 
Then, asymptotically almost surely and uniformly over $p$, any coloring of $G^r(n,p)$ with $x\leq \frac{n}{r}$ colors  
contains a bichromatic balanced triangle. 
\end{lemma}

\begin{proof}
Considering a particular coloring of $G^r{(n,p)}$ with $x$ colors, from all possible $x^n$ colorings, we show that the probability of the event $E$ of $G^r(n,p)$ not having a bichromatic balanced triangle is $o(n^{-n}).$ 
Let $V_1,\dots, V_x$ denote color classes. 
We partition each $V_i$, $1\leq i\leq x$ into equal number of parts of order $a$ and $b$, where $(a,b)$ have value $(\lfloor 3r/4 \rfloor, \lceil 3r/4 \rceil)$ and 
$(\lfloor (3r+1)/4 \rfloor, \lceil (3r+1)/4 \rceil)$, when $r$ is even and odd, respectively, and call the leftover $L_i$. 
A balanced triangle has order $a+b$, that is $3r/2$ and $(3r+1)/2$, when $r$ is even and odd, respectively. 
Let $X_1,\dots, X_{u'}$ and $Y_1,\dots, Y_{u'}$ be the parts of order $a$ and $b$, respectively, obtained from this partition. 
%Next, we remove some more pairings from these leftovers so that as many vertices as possible from each $V_i$ are used in these pairings of order $(a,b).$ 
Then, we do the same for the leftovers, by removing possibly a part of order $a$ or $b$ from the $L_i$'s, $1\leq i\leq x$, call these parts $X_{u'+1},\dots, X_u$ and $Y_{u'+1},\dots, Y_u$, respectively. 
Next, we aim 
to construct as many edge-disjoint balanced triangles as possible on the vertex subsets $X_j\cup Y_k$, $1\leq j,k\leq u.$ 

 Let $s=u^2$, that is the number of pairings of the parts of order $a$ and $b$, as described above.  Let $t$ be the maximum possible number of edge-disjoint balanced triangles on $a+b$ vertices. 
Let $E'$ be the event of $G^r(n,p)$ not having any bicolored or monochromatic balanced triangle induced by the vertex sets $X_j\cup Y_k$, $1\leq j,k\leq u.$ 
Let $E$ be the event of $G^r(n,p)$ not having any bicolored or monochromatic balanced  triangle. Since $E'\subseteq E$, we have 
\begin{equation}\label{eqn:event_bound}
    \Pr(E) \leq \Pr(E') \leq (1 - p^3)^{st}.
\end{equation}
In the following, we show a lower bound for $s$ and $t,$ respectively.

With the process below, we remove some more vertices from each $L_i$ to use in the partition to minimize the amount of leftover vertices.  
We group $L_1,\dots,L_x$ into three parts: 
$B:=\{i: |L_i|\geq b\}$, 
$A:=\{i: |L_i|=a\}$, 
$R:=\{i: |L_i|< a\}$. 
First, assume that $r$ is odd, hence $a\neq b.$
All vertices in $L_i$ with $i\in R$ are removed. 
If $|B|> |A|$, then we move some elements of $B$ to $A$ so 
that $|A|=|B|$ or $|B|=|A|+1.$ 
If $|B|<|A|$ ($|B|\geq |A|$, resp.), 
then we remove $|B|$ ($|A|$, resp.) many parts with order $a$ and $b$ 
from each $L_i$ and $L_j$, $i\in A$ and $j\in B$, respectively, 
and include these parts into the partition above. 
This yields a partition with equal number of parts 
of order $a$ and $b$. 
We remove the leftover from each $L_i$, which results 
in removing at most $a$ vertices from each $L_i.$ 
Only in the case, $|B|=|A|+1$, there is a single $L_i$ with 
$i\in B$ that has no contribution to the partition.  
It is possible that upto $2a$ vertices are removed from this 
particular $L_i.$ 
If $r$ is even, then $a=b$ and $A\subseteq B.$ In this case, let 
$B'=B-A$, $A'=A$ and repeat the same procedure above for $B'$ and $A'$.

The above procedure, removes at most $a$ vertices from each color class, except possibly one color class, from which upto $2a$ vertices may be removed. Hence, the total number of vertices removed is at most $(x+1)(3r+1)/4.$ 
%Let $b$ be the number of pairings of these sets with union $3r/2$ vertices. 
Thus,    
\begin{equation}\label{eqn:triangle_bound}
    s \geq \left(\frac{n-\frac{(x+1)(3r+1)}{4}}{\frac{3r+1}{2}}\right)^2 \geq  %\left(\frac{1}{2}\frac{4n}{3r+1}-(\frac{n}{r}+1)\right)^2 \geq  
    \left(\frac{1}{2}\frac{n(r-1)}{r(3r+1)}-1\right)^2 \geq
    \frac{n^2}{200r^2} \quad \text{for} \quad x \leq \frac{n}{r},
\end{equation}
as $(r-1)/(3r+1)\geq \frac{1}{7}$ for $r\geq 2.$ (Note that, we obtain a higher lower bound for even values of $r$, the bound above is for general $r$.)

To find a lower bound on $t$, we consider the number of edge-disjoint balanced triangles on a set of $a+b$ vertices. 
If $r$ is even, the total number of balanced triangles on $3r/2$ vertices is  $\frac{1}{3!}\binom{3r/2}{r}\binom{r}{r/2}$ and each edge appears in exactly $\frac{1}{2}\binom{r}{r/2}$ triangles. Thus, 
 \begin{equation}\label{eqn:lower_bound_aeven}
    t \geq \frac{\frac{1}{6} \binom{3r/2}{r} \binom{r}{r/2}}{\frac{3}{2} \binom{r}{r/2} } = 
    %\frac{1}{9} \frac{\binom{3r/2}{r} \binom{r}{r/2}}{\binom{r}{r/2}} = 
    \frac{1}{9} \binom{3r/2}{r} \geq \frac{1}{9} \cdot \left(\frac{3}{2}\right)^r.
\end{equation}
If $r$ is odd, not every edge in the balanced triangle has the same role. 
Thus, the total number of balanced triangles on $(3r+1)/2$ vertices is  $\frac{r+1}{4}\binom{\frac{3r+1}{2}}{r}\binom{r}{\frac{r+1}{2}}$ and 
each edge appears in exactly 
%$\binom{r}{(r-1)/2}\frac{r+1}{2}+\binom{r}{(r+1)/2}\frac{r+1}{2}$
$(r+1)\binom{r}{(r+1)/2}$ triangles. Thus, 
 \begin{equation}\label{eqn:lower_bound_aodd}
    t \geq 
    \frac{\frac{1}{2}\binom{(3r+1)/2}{r}\binom{r}{(r+1)/2}\frac{r+1}{2}}
    {3(r+1)\binom{r}{(r+1)/2}} = 
    %\frac{1}{9} \frac{\binom{3r/2}{r} \binom{r}{r/2}}{\binom{r}{r/2}} = 
    \frac{1}{12} \binom{(3r+1)/2}{r} \geq 
    \frac{1}{12} \cdot \left(\frac{3}{2}\right)^r.
\end{equation}
So, we use the bound in \eqref{eqn:lower_bound_aodd} for general $r$ as below. 
Following from \eqref{eqn:event_bound}, and substituting lower bounds \eqref{eqn:triangle_bound} and \eqref{eqn:lower_bound_aodd} for $s$ and $t$, we have:
\begin{multline*}
    Pr(E')\leq (1 - p^3)^{st} \leq e^{-p^3st} \leq \exp\left\{-p^3  \frac{n^2}{200r^2}  \frac{1}{12} {\left(\frac{3}{2}\right)}^r\right\} \leq \\
    \exp\left\{-p^3 \frac{n^2}{r^2} \frac{3^{r-1}}{25\cdot 2^{r+5}}\right\} 
    %\leq \exp\left\{-2n\ln(n)\right\} 
    = o(n^{-n}) \text{ for $p \geq {\left(\frac{r^2\ln(n)}{n} \cdot \frac{25\cdot 2^{r+6}}{3^{r-1}}\right)}^\frac{1}{3}$}.
\end{multline*}
\end{proof}

%\begin{proof}[Proof of Theorem \ref{thm:lwrbnd}]

We let $D = \binom{n-1}{r-1}$ denote the maximum possible degree of a vertex in $G^r(n,p).$ We choose $n$ sufficiently large with respect to $d$ as:
\begin{equation}\label{d_bound}
    a_r\frac{\ln{n}}{n}D^3 \leq d^3 \leq 2a_r \frac{\ln n}{n} D^3,
\qquad \text{where} \quad a_r = \frac{25\cdot 2^{r+9}r^2}{3^{r-1}}.
\end{equation}

Let $p=\frac{d-s}{D},$ where $s=\sqrt{2d\ln{(Dn)}}.$ By Chernoff Bound in \eqref{eq:chernoff} and the substitution for $s$ and $p$, we observe the following for $\Delta(G^r(n,p))$. 
\begin{equation}\label{eq:delta-bound}
\begin{split}
       \Pr(\Delta(G^r(n,p)) > d) &\leq n \Pr(\text{BIN}(D,p) \geq d) \\
  & \leq  n \exp \left( -\frac{s^2}{2d} \right) \leq 
  \left(\frac{r-1}{n-1}\right)^{r-1}. 
\end{split}
\end{equation}
To have the proper range for $p=\frac{d-\sqrt{2d\ln{(Dn)}}}{D}$, 
we assume in accordance with \eqref{d_bound} that $d$ is large enough such that $p^3D^3\geq (\frac{d}{2})^3.$ This ensures that $p$ satisfies the requirement in Lemma \ref{lem_linearpath} as follows, where the last inequality follows by \eqref{d_bound}.  
\[
p^3 \geq (\frac{d}{2D})^3 \geq \frac{25\cdot 2^{r+6}r^2}{3^{r-1}}\frac{\ln n}{n}.
\]
By Lemma \ref{lem_linearpath}, we can assume the following for sufficiently large $d.$ 
\[
\Pr\left(\chi_s(G^r(n, p)) \geq x\right) \geq \frac{3}{4},
\]
This, together with \eqref{eq:delta-bound} gives:
\begin{equation}%\label{intersection} --- no reference to this ----
  \Pr\left(\chi_s(G^r(n, p)) \geq x, \Delta(G^r(n,p)) \leq d\right) \geq \frac{3}{4} - \left(\frac{r-1}{n-1}\right)^{r-1} > 0. 
\end{equation}
This ensures the existence of a random graph $G=G^r(n,p)$ with $\chi_s(G)\geq x$ and maximum degree at most $d.$ 
By our assumption on $x$ in Lemma \ref{lem_linearpath},  
$\chi_s^r(d) \geq x \geq \frac{n}{r}.$ 
By \eqref{d_bound}, we have 
\[
d^3 \leq \frac{25\cdot 2^{r+10}r^2}{3^{r-1}} \frac{\ln n}{n} D^3 \leq 
\frac{25\cdot 2^{r+10}r^2e^{3r-3}}{3^{r-1}(r-1)^{3r-3}}n^{3r-4}\ln n.
\]
Moreover, by the lower bound on $d$ in \eqref{d_bound}, we have 
$\ln n \leq \frac{3}{3r-4}\ln d\leq \frac{3}{2}\ln d$ for $r\geq 2$ and 
sufficiently large $n.$ These together yield 
\[
\chi_s^r(d) \geq x \geq \frac{n}{r} \geq 
\left(\frac{3^{r-2}(r-1)^{3r-3}}{r^{3r-2}25\cdot 2^{r+9}e^{3r-3}}\frac{d^3}{\ln d}\right)^{\frac{1}{3r-4}}
\]
%\end{proof}

\section{Concluding Remarks}

Lemma \ref{lem_linearpath} can be generalized to any $r$-uniform subhypergraph with at most $2r-2$ vertices besides the balanced triangle. Thus, 
the method used for proving the lower bound in Theorem~\ref{thm:lwrbnd} can be applied for any coloring forbidding bicolored (and monochromatic) members of any family $\FF$ that contains a member with at most $2r-2$ vertices. 

There is a large gap in the exponent of $d$ observed in the upper and lower bound for $\chi_s^r(d)$ and $\chi_a^r(d)$ in Section~\ref{sec:appstar}. We suspect that the exact value of these parameters is closer to the lower bound in Corollary~\ref{cor:lowerbnd}.  
\bibliographystyle{plain} 

\bibliography{submission_arxiv2026}

@article{bohman2010coloring,
  title={Coloring H-free hypergraphs},
  author={Bohman, Tom and Frieze, Alan and Mubayi, Dhruv},
  journal={Random Structures \& Algorithms},
  volume={36},
  number={1},
  pages={11--25},
  year={2010},
  publisher={Wiley Online Library}
}

@article{cooper2015list,
  title={List coloring triangle-free hypergraphs},
  author={Cooper, Jeff and Mubayi, Dhruv},
  journal={Random Structures \& Algorithms},
  volume={47},
  number={3},
  pages={487--519},
  year={2015},
  publisher={Wiley Online Library}
}

@article{cooper2016coloring,
  title={Coloring sparse hypergraphs},
  author={Cooper, Jeff and Mubayi, Dhruv},
  journal={SIAM Journal on Discrete Mathematics},
  volume={30},
  number={2},
  pages={1165--1180},
  year={2016},
  publisher={SIAM}
}

@article{cooper2017sparse,
  title={Sparse hypergraphs with low independence number},
  author={Cooper, Jeff and Mubayi, Dhruv},
  journal={Combinatorica},
  volume={37},
  pages={31--40},
  year={2017},
  publisher={Springer}
}

@inproceedings{erdHos1973problems,
  title={Problems and results on 3-chromatic hypergraphs and some related questions},
  author={Erd{\H{o}}s, Paul and Lov{\'a}sz, L{\'a}szl{\'o}},
  booktitle={in Infinite and Finite Sets (A. Hajnal et al., eds)},
  year={1975},
  organization={North-Holland, Amsterdam}
}

@article{frieze2011randomly,
  title={Randomly coloring simple hypergraphs},
  author={Frieze, Alan and Melsted, P{\'a}ll},
  journal={Information Processing Letters},
  volume={111},
  number={17},
  pages={848--853},
  year={2011},
  publisher={Elsevier}
}

@article{frieze2008chromatic,
  title={On the chromatic number of simple triangle-free triple systems},
  author={Frieze, Alan and Mubayi, Dhruv},
  journal={the electronic journal of combinatorics},
  volume={15},
  number={1},
  pages={R121},
  year={2008}
}

@article{frieze2013coloring,
  title={Coloring simple hypergraphs},
  author={Frieze, Alan and Mubayi, Dhruv},
  journal={Journal of Combinatorial Theory, Series B},
  volume={103},
  number={6},
  pages={767--794},
  year={2013},
  publisher={Elsevier}
}

@article{kostochka2012conflict,
  title={Conflict-free colourings of uniform hypergraphs with few edges},
  author={Kostochka, Alexander and Kumbhat, Mohit and {\L}uczak, T},
  journal={Combinatorics, Probability and Computing},
  volume={21},
  number={4},
  pages={611--622},
  year={2012},
  publisher={Cambridge University Press}
}

@article{kostochka2009coloring,
  title={Coloring uniform hypergraphs with few edges},
  author={Kostochka, Alexandr V and Kumbhat, Mohit},
  journal={Random Structures \& Algorithms},
  volume={35},
  number={3},
  pages={348--368},
  year={2009},
  publisher={Wiley Online Library}
}

@article{li2022chromatic,
  title={The chromatic number of triangle-free hypergraphs},
  author={Li, Lina and Postle, Luke},
  journal={arXiv preprint arXiv:2202.02839},
  year={2022}
}

@inproceedings{erd1975problems,
  title={Problems and results on 3-chromatic hypergraphs and some related questions, in “Infinite and Finite Sets”(A. Hajnal et al., Eds.)},
  author={Erd{\H{o}}s, P and Lov{\'a}sz, L},
  booktitle={Colloq. Math. Soc. J. Bolyai},
  volume={11},
  pages={609},
  year={1975}
}

@article{albertson2004coloring,
  title={Coloring with no $2 $-Colored $ P\_4 $'s},
  author={Albertson, Michael O and Chappell, Glenn G and Kierstead, Hal A and K{\"u}ndgen, Andr{\'e} and Ramamurthi, Radhika},
  journal={the electronic journal of combinatorics},
  pages={R26--R26},
  year={2004}
}

@article{grunbaum1973acyclic,
  title={Acyclic colorings of planar graphs},
  author={Gr\"{u}nbaum, Branko},
  journal={Israel journal of mathematics},
  volume={14},
  number={4},
  pages={390--408},
  year={1973},
  publisher={Springer}
}

@incollection{kostochka1978upper,
  title={Upper bounds of chromatic functions of graphs},
  author={Kostochka, Alexandr V},
  booktitle={Doct. Thesis},
  year={1978},
  publisher={Novosibirsk}
}

@article{fertin2004star,
  title={Star coloring of graphs},
  author={Fertin, Guillaume and Raspaud, Andr{\'e} and Reed, Bruce},
  journal={Journal of Graph Theory},
  volume={47},
  number={3},
  pages={163--182},
  year={2004},
  publisher={Wiley Online Library}
}

@article{alon1991acyclic,
  title={Acyclic coloring of graphs},
  author={Alon, Noga and McDiarmid, Colin and Reed, Bruce},
  journal={Random Structures \& Algorithms},
  volume={2},
  number={3},
  pages={277--288},
  year={1991},
  publisher={Wiley Online Library}
}

@article{esperet2013acyclic,
  title={Acyclic edge-coloring using entropy compression},
  author={Esperet, Louis and Parreau, Aline},
  journal={European Journal of Combinatorics},
  volume={34},
  number={6},
  pages={1019--1027},
  year={2013},
  publisher={Elsevier}
}

@article{ndreca2012improved,
  title={Improved bounds on coloring of graphs},
  author={Ndreca, Sokol and Procacci, Aldo and Scoppola, Benedetto},
  journal={European Journal of Combinatorics},
  volume={33},
  number={4},
  pages={592--609},
  year={2012},
  publisher={Elsevier}
}

@inproceedings{goncalves:lirmm-01233456,
  TITLE = {{Entropy compression method applied to graph colorings}},
  AUTHOR = {Gon{\c c}alves, Daniel and Montassier, Micka{\"e}l and Pinlou, Alexandre},
  BOOKTITLE = {{ICGT: International Colloquium on Graph Theory and Combinatorics}},
  ADDRESS = {Grenoble, France},
  YEAR = {2014},
  MONTH = Jun,
  HAL_ID = {lirmm-01233456},
  HAL_VERSION = {v1},
}

@book{janson2011random,
  title={Random graphs},
  author={Janson, Svante and Luczak, Tomasz and Rucinski, Andrzej},
  year={2011},
  publisher={John Wiley \& Sons}
}

\end{document}